\documentclass[12pt]{article}
\usepackage[a4paper,margin=1.3in]{geometry}
\usepackage{amssymb}
\usepackage{amsthm}
\usepackage{amsmath}
\usepackage{amsfonts}
\usepackage{graphicx}
\usepackage{float}
\usepackage[all]{xy}
\usepackage{scalerel}

\def\hurw{\mathop{\mbox{\textsl{H}}}\nolimits}
\def\ker{\mathop{\mbox{\textsl{Ker}}}\nolimits}
\def\ann{\mathop{\mbox{\textsl{Ann}}}\nolimits}

\def\exp{\mathop{\mbox{\textsf{exp}}}\nolimits}

\newcommand{\Opa}{\mathfrak{A}}

\newcommand{\C}{\mathbb{C}}
\newcommand{\Z}{\mathbb{Z}}

\newcommand{\N}{\mathbb{N}}

\newcommand{\R}{\mathbb{R}}

\newcommand{\F}{\mathbb{F}}

\DeclareMathOperator*{\bcast}{\scalerel*{\odot}{\sum}}

\DeclareMathOperator*{\bbox}{\scalerel*{\boxplus}{\sum}}

\newtheorem{theorem}{Theorem}
\newtheorem{lemma}{Lemma}
\newtheorem{definition}{Definition}
\newtheorem{proposition}{Proposition}
\newtheorem{corollary}{Corollary}
\newtheorem{example}{Example}

\begin{document}
\title{\textbf{Ring of flows of one-dimensional differential equations}}
\author{Ronald Orozco L\'opez}

\newcommand{\Addresses}{{

\textit{E-mail address}, R.~Orozco: \texttt{rj.orozco@uniandes.edu.co}
  
}}

\maketitle

\begin{abstract}
In this article the ring of flows of autonomous differential equations of order one on integral domains is constructed. First we build the autonomous ring $\Opa(\hurw_{R}[[x]])$ and then its structure is studied. Next, we build the ring of formal exponential generating series of the ring
$\Opa(\hurw_ {R} [[x]]) $, where it is possible to find solutions of differential equations of order one when the vector field of the system can be decomposed in sum or product of functions.
\end{abstract}

\section{Introducci\'on}
A very important problem in the theory of autonomous differential equations is to find exact solutions of these. In the one-dimensional case the solutions of $\phi^{\prime}=f(\phi)$
can be found by using the method of separation of variables. However, it is not always possible to find explicit solutions to these equations by this method. In this article we propose a different method that allows not only to find solutions to these equations but also to give a ring structure to the set of one-dimensional dynamic systems. This type of solutions are found if the vector field $f(x)$ is taken from the Hurwitz ring of power series $\hurw_{R}[[x]]$, where $R$ is a integral domain.

In article [5] we make a first approach to this type of solutions. In particular we find an analytical solution to the autonomous equation $\phi^{(k)}=f(\phi)$ of order $k$ in terms of autonomous polynomials. In this paper the equations will be limited to order one.

Autonomous polynomials are constructed using Bell polynomials. Bell polynomials are a very useful tool in mathematics to represent the $n$-th derivative of the composition of functions (see [3]). Indeed, let $f$ and $g$ be two analytic functions with representation
in power series $\sum a_{n}\frac{x^{n}}{n!}$ and $\sum b_{n}\frac{x^{n}}{n!}$, respectively, with $a_{n},b_{n}\in\C$. Then
\begin{equation}
f(g(x))= f(b_{0}) + \sum_{n=1}^{\infty}Y_{n}(b_{1},...,b_{n};a_{1},...,a_{n})\frac{x^{n}}{n!},
\end{equation}
where $Y_{n}$ is the $n$-th Bell polynomial. For example, it is well known that
\begin{eqnarray*}
Y_{1}(b_{1};a_{1})&=&a_{1}b_{1},\\
Y_{2}(b_{1},b_{2};a_{1},a_{2})&=&a_{1}b_{2}+a_{2}b_{1}^{2},\\
Y_{3}(b_{1},b_{2},b_{3};a_{1},a_{2},a_{3})&=&a_{1}b_{3}+a_{2}(3b_{1}b_{2})+a_{3}b_{1}^{3}.
\end{eqnarray*}

Bell's polynomials can be expressed explicitly by using the formula of Faá di Bruno ([3])
\begin{equation}\label{faa}
Y_{n}(b_{1},...,b_{n};a_{1},...,a_{n})=\sum_{k=1}^{n}B_{n,k}a_{k},
\end{equation}
where
\begin{equation}\label{partialBell}
B_{n,k}=\sum_{\vert p(n)\vert=k}\frac{n!}{j_{1}!j_{2}!\cdots j_{n}!}\left[\frac{b_{1}}{1!}\right]^{j_{1}}\left[\frac{b_{2}}{2!}\right]^{j_{2}}\cdots \left[\frac{b_{n}}{n!}\right]^{j_{n}},
\end{equation}
and the sum runs over all partitions $p(n)$ of $n$, that is, $n=j_{1}+2j_{2}+\cdots+nj_{n}$, $j_{h}$ denotes the number of parts of size $h$ y $\vert p(n)\vert=j_{1}+j_{2}+\cdots+j_{n}$ is the length 
of the partition $p(n)$.

Suppose $ f $ with power series representation and let
\begin{equation}\label{autonoma}
\begin{cases}
u^{\prime} = f(u)\\
u(0)=x
\end{cases}
\end{equation}
be the autonomous differential equation with initial value problem.

In [5] autonomous polynomials were introduced, which are the coefficients of the analytical solution of the equation (\ref{autonoma}) expressed as a function of the initial value in the following way
\begin{small}
\begin{equation}\label{polinomio_autonomo}
\begin{cases}
A_{1}(f(x))=f(x),\\
A_{n+1}(f(x),f^{\prime}(x),...,f^{(n)}(x))=Y_{n}(A_{1}(f(x)),...,A_{n}(f(x));f^{\prime}(x),...,f^{(n)}(x))
\end{cases}
\end{equation}
\end{small}

From (\ref{polinomio_autonomo}) we can see that each autonomous polynomial $A_{n}$ is defined on the sequence $(f^{(n)}(x))_{n\geq0}$ of derivatives of $f(x)\in C_{\infty}(\R)$, where $C_{\infty}(\R)$ 
is the set of all smoothness functions in $\R$. In this article, autonomous differential equations of order one with functions in $\hurw_{R}[[x]]$, the Hurwitz ring of exponential generating functions, are studied. The article is divided as follows. In the second section the ring $(\hurw_{R}[[x]],+,\cdot) $ and its structure are studied. In the third section we define the autonomous operator and construct the autonomous ring $(\Opa(\hurw_{R}[[x]]),\boxplus,\odot)$ that will serve as the basis for defining operations with flows of dynamical systems. Dynamic systems on rings are defined in section four. In this section it is proved that the flow of a dynamic system is a $R$-module acting on the phase space of the system. Finally, in section five the ring of generating functions of sequences in the autonomous ring is constructed. With this ring it is possible to define the one-dimensionals flows ring with which it will be possible to define flow operations. That is, we can decompose the flow of a dynamic system into simpler flows.

Throughout this paper $R$ will be a ring of characteristic zero. Also, $\N$ will be the monoid of non-negative integers and $\N_{1}$ will be the semigroup of positive integers.

\section{The Hurwitz ring $(\hurw_{R}[[x]],+,\cdot)$}

Throughout this paper let $(R,+,\cdot)$ denote an integral domain. Let $\hurw_{R}=\{(a_{n})_{n\in\N}: a_{n}\in R\}$ denote the set of sequences with elements in the ring $R$. Take $\textbf{a}=(a_{n})_{n\in\N}$ and $\textbf{b}=(b_{n})_{n\in\N}$ in $\hurw_{R}$. Define the sum in $\hurw_{R}$ as 
$\textbf{a}+\textbf{b}=(a_{n}+b_{n})_{n\in\N}$. Let $\cdot$ denote the Hadamard product of $\textbf{a}$ and $\textbf{b}$ given by $\textbf{a}\cdot\textbf{b}=(a_{n}\cdot b_{n})_{n\in\N}$, that is, the product in $\hurw_{R}$ is a componentwise product. Then $(H_{R},+,\cdot)$ is a ring with unit element $\textbf{1}=(1,1,1,...)$. An element $\textbf{a}$ is invertible in $\hurw_{R}$ with respect to the product $\cdot$ if and only if $a_{n}\in R^{*}$ for all $n\geq0$.

Now let $*$ denote the Hurwitz product in $\hurw_{R}$ as
\begin{equation}
\textbf{a}\ast\textbf{b}=\left(\sum_{h=0}^{n}\binom{h}{n}a_{h}b_{n-h}\right)_{n\in\N}.
\end{equation}

Then $(\hurw_{R},+,\ast)$ is a ring with unit element $\textbf{e}=(1,0,0,...)$. An element 
$\textbf{a}=(a_ {n})_{n\in\N}$ in $\hurw_{R}$ is invertible with respect to $\ast$ if and only if
$a_{0}\in R^{*}$. Let 
$$\hurw_{R}^{*}=\{\textbf{a}^{-1}:\textbf{a}\in\hurw_{R}\}$$ 
denote the set of invertible elements in $(\hurw_{R},+,*)$, where $\textbf{a}^{-1}=\textbf{b}=(b_{n})_{n\in\N}$ with

\begin{equation}
\begin{cases}
b_{0}=a_{0}^{-1}\\
b_{n}=-a_{0}^{-1}\sum_{h=1}^{n}\binom{n}{h}a_{h}b_{n-h},\ n\geq1.
\end{cases}
\end{equation}

Now denote $\hurw_{R}[[x]]$ the set of formal power series of the form $\sum_{n=0}^{\infty}a_{n}\frac{x^{n}}{n!}$ with coefficients in $R$. It is clear that 
$(\hurw_{R}[[x]],+,\cdot)$ is a ring with sum and product of ordinary series, that is,

$$f(x)+g(x)=\sum_{n=0}^{\infty}(a_{n}+b_{n})\frac{x^{n}}{n!},$$  
$$f(x)\cdot g(x)=\sum_{n=0}^{\infty}\sum_{k=0}^{n}\binom{n}{k}a_{k}b_{n-k}\frac{x^{n}}{n!},$$ 

where $f(x)=\sum_{n=0}^{\infty}a_{n}\frac{x^{n}}{n!}$, 
$g(x)=\sum_{n=0}^{\infty}b_{n}\frac{x^{n}}{n!}\in\hurw_{R}[[x]]$. The ring $\hurw_{R}[[x]]$ is
known as ring of Hurwitz of power series (ve\'ase [4]).

Let $\rho_{x}:(\hurw_{R},+,*)\rightarrow(\hurw_{R}[[x]],+,\cdot)$ be an isomorphism given by
$$\rho_{x}(\textbf{a})=\rho_{x}((a_{n})_{n\in\N})=\sum_{n=0}^{\infty}a_{n}\frac{x^{n}}{n!}=f(x).$$

If $\rho_{x}(\textbf{a})=f(x)$ y $\rho_{x}(\textbf{b})=g(x)$, then

$$\rho_{x}(\textbf{a}\ast\textbf{b})=\rho_{x}(\textbf{a})\cdot\rho_{x}(\textbf{b})$$
and
$$\rho_{x}(\textbf{a}^{-1})=\frac{1}{\rho_{x}(\textbf{a})}=(\rho_{x}(\textbf{a}))^{-1}$$

Let $\delta$ be a derivation on $\hurw_{R}$ define by

$$\delta((a_{n})_{n\in\N})=(a_{n+1})_{n\in\N}$$

with $\delta(\textbf{a}*\textbf{b})=\delta(\textbf{a})\ast\textbf{b}+\textbf{a}*\delta(\textbf{b})$
and let $\delta_{x}\equiv\frac{d}{dx}$ be a derivation on $\hurw_{R}[[x]]$ given by 

$$\delta_{x}\left(\sum_{n=0}^{\infty}a_{n}\frac{x^{n}}{n!}\right)=\sum_{n=0}^{\infty}a_{n}\frac{\delta_{x}x^{n}}{n!}=\sum_{n=0}^{\infty}a_{n+1}\frac{x^{n}}{n!}.$$

Then

\begin{equation}\label{relation_delta_mu}
\delta_{x}\rho_{x}(\textbf{a})=\rho_{x}(\delta\textbf{a}),
\end{equation}

that is, the following diagram commutes

\begin{equation}
\xymatrix{
 \hurw_{R} \ar[d]^{\rho_{x}} \ar[r]^{\delta} & \hurw_{R} \ar[d]^{\rho_{x}}\\
 \hurw_{R}[[x]] \ar[r]^{\delta_{x}} & \hurw_{R}[[x]] 
}
\end{equation}

It is easy to see that $\rho_{x}(\ker\delta)=\ker\delta_{x}$, that is, the constant ring of 
$\hurw_{R}[[x]]$ is the image by $\rho_{x}$ of the constant ring of $\hurw_{R}$.

\section{The Ring Autonomous $(\Opa(\hurw_{R}[[x]]),\boxplus,\odot)$}

In this section the autonomous ring is constructed. This is a ring of sequences of autonomous polynomials defined on the ring $\hurw_{R}[[x]] $. We first define the autonomous operator and study its properties. We then use this operator to construct the binary operations of the autonomous ring. We start with the following definition

\begin{definition}\label{def_opa}
Let $\hurw_{S}$ be the ring of sequences on the ring $S=\hurw_{R}[[x]]$. We define the 
\textbf{autonomous operator} $\Opa$ as the linear map
$\Opa:\hurw_{R}[[x]]\rightarrow\hurw_{S}$ defined by 
\begin{equation}
\Opa(f(x))=(A_{0}(f(x)),A_{1}(f(x)),A_{2}(f(x)),A_{3}(f(x)),...)
\end{equation}
where the $A_{n}$ are recursively defined by
\begin{eqnarray*}
A_{0}(f(x))&=&x,\\
A_{1}(f(x))&=&f(x),\\
A_{n+1}(f(x))&=&Y_{n}(A_{1}(f(x)),A_{2}(f(x))...,A_{n}(f(x));\delta^{1}f(x),...,\delta^{n}f(x)),
\end{eqnarray*}
$n\geq 1$.\\
The polynomials $A_{n}$ in the indeterminate $\delta^{0}f,\delta^{1}f,...,\delta^{n-1}f$ will be called \textbf{autonomous polynomials}.
\end{definition}

Some autonomous polynomials $A_{n}$ are
\begin{eqnarray*}
A_{2}(f(x)) &=& A_{1}(f(x))\delta^{1}f(x)=f(x)\delta f(x),\\
A_{3}(f(x)) &=& A_{2}(f(x))\delta f(x)+A_{1}^{2}(f(x))\delta^{2}f(x),\\
&=&f(x)(\delta f(x))^{2}+f^{2}(x)\delta^{2}f(x),\\
A_{4}(f(x)) &=& A_{3}(f(x))\delta^{1}f(x)+3A_{1}(f(x))A_{2}(f(x))\delta^{2}f(x)+A_{1}^{3}(f(x))\delta^{3}f(x),\\
&=&f(\delta f(x))^{3}+4f^{2}(x)\delta^{1}f(x)\delta^{2}f(x)+f^{3}(x)\delta^{3}f(x).\\
\end{eqnarray*}

In the following Lemma we will see how to construct the autonomous polynomial $A_{n}(f(x))$ from the polynomial $A_{n-1}(f(x))$

\begin{lemma}\label{lemma_delta_opa}
Take $f(x)\in\hurw_{R}[[x]]$. Then
\begin{equation}
A_{n+1}(f(x))=f(x)\delta(A_{n}(f(x))),
\end{equation}
con $n\geq1$.
\end{lemma}
\begin{proof}
By definition $A_{2}(f)=f\delta_{}(f)$. Now suppose by induction hypothesis that
$A_{n}(f)=f\delta(A_{n-1}(f))$. Put $A_{n}\equiv A_{n}(f)$. On the one hand, we have
\begin{eqnarray*}
A_{n+1}&=&\delta_{}fA_{n}+\delta_{}^{2}f\{A_{1}A_{n-1}+A_{2}A_{n-2}+\cdots\}+\cdots\\
&=&\delta_{}(f)f\delta(A_{n-1})+\delta_{}^{2}(f)\{ff\delta(A_{n-2})+f\delta(A_{1})A_{n-2}+\cdots\}+\cdots\\
&=&f[\delta_{}(f)\delta(A_{n-1})+\delta_{}^{2}(f)\{f\delta(A_{n-2})+\delta(A_{1})A_{n-2}+\cdots\}+\cdots]
\end{eqnarray*} 
On the other hand,
\begin{eqnarray*}
f\delta(A_{n})&=&f\delta(\delta_{}(f)A_{n-1}+\delta_{}(f)\{A_{1}A_{n-2}+A_{2}A_{n-3}+\cdots\}+\cdots)\\
&=&f[\delta_{}(f)\delta(A_{n-1})+\delta_{}^{2}(f)A_{n-1}+\delta_{}^{2}(f)\delta(A_{1})A_{n-2}+\delta_{}^{2}(f)A_{1}\delta(A_{n-2})+\cdots]\\
&=&f[\delta_{}(f)\delta(A_{n-1})+\delta_{}^{2}(f)\{f\delta(A_{n-2})+\delta(A_{1})A_{n-2}+\cdots\}+\cdots]
\end{eqnarray*}
Now by comparing the results we obtain $A_{n+1}=f\delta(A_{n})$.
\end{proof}

From the above we obtain the following corollary

\begin{corollary}\label{coro_delta_aut}
Take $f(x)\in\hurw_{R}[[x]]$. Then
\begin{equation}
A_{n}(f(x))=\overbrace{f(x)\delta(\cdots f(x)\delta(f(x)))}^{n}.
\end{equation}
\end{corollary}

Now we will define operations in $\Opa(\hurw_{R}[[x]])$ that allow to give it a ring structure.

\begin{definition}
We define the sum $\boxplus$ in $\Opa(\hurw_{R}[[x]])$ as follows
\begin{equation}
\Opa(f)\boxplus\Opa(g)=(B_{n}(f,g))_{n\in\N}
\end{equation}
where
\begin{eqnarray}
B_{0}(f,g)&=&x\\
B_{1}(f,g)&=&A_{1}(f)+A_{2}(g)\\
B_{n}(f,g)&=&A_{n}(f)+A_{n}(g)+H_{n}(f,g),\ \ \ n\geq2
\end{eqnarray}
where $H_{n}(f,g)$ is a polynomial in the variables $f,f^{\prime},...,f^{(n-1)}$ and 
$ g,g^{\prime},...,g^{(n-1)}$ fulfilling
\begin{equation}\label{eqn_Hn}
H_{n+1}(f,g)=f\delta A_{n}(g)+g\delta A_{n}(f)+(f+g)\delta H_{n}(f,g)
\end{equation}
\end{definition}

\begin{lemma}\label{lemma_ope_+}
Take $f(x),g(x)\in\hurw_{R}[[x]]$. Then
\begin{equation}
\Opa(f(x))\boxplus\Opa(g(x))=\Opa(f(x)+g(x))
\end{equation}
\end{lemma}
\begin{proof}
Finding the autonomous polynomials for the sum $ f + g $ we obtain
\begin{equation}
A_{n}(f+g)=A_{n}(f)+A_{n}(g)+H_{n}(f,g).
\end{equation}
where $H_{n}(f,g)$ is a polynomial in the variables $f,f^{\prime},...,f^{(n-1)}$ and 
$g,g^{\prime},...,g^{(n-1)}$. If $H_{n}(f,g)$ satisfies (\ref{eqn_Hn}) and if we use the Lemma \ref{lemma_delta_opa}, we have
\begin{eqnarray*}
A_{n+1}(f+g)&=&(f+g)\delta A_{n}(f+g)\\
&=&(f+g)\delta[A_{n}(f)+A_{n}(g)+H_{n}(f,g)]\\
&=&f\delta A_{n}(f)+g\delta A_{n}(g)+f\delta A_{n}(g)\\
&&+g\delta A_{n}(f)+(f+g)H_{n}(f,g)\\
&=&A_{n+1}(f)+A_{n+1}(g)+H_{n+1}(f,g)
\end{eqnarray*}
Then our affirmation is followed.
\end{proof}

We want to find the additive inverse of the element $\Opa(f)$ in $\Opa(\hurw_{R}[[x]])$. For this we have first the definition

\begin{definition}
Take $a\in R-\{0\}$. Let $\exp(a)$ denote the expansion in non-negative integer powers from $a$
\begin{equation}
\exp(a)=(1,a,a^{2},a^{3},...).
\end{equation}
For $a=0$ define
\begin{equation}
\exp(0)=(1,0,0,0,...).
\end{equation}
\end{definition}

The reason for denoting the expansion in non-negative integer powers from $a$ as $\exp(a)$ is because 
$\rho_{x}\exp(a)=e^{ax}$ is the exponential function knowed. The following result tells us that the set of all expansions $\exp(a)$ form a commutative ring with unity.

\begin{theorem}
Let $\exp(R)$ denote the set of all expansions $\exp(a)$, $a\in R$. Then 
$(\exp(R),\ast,\cdot)$ is a commutative ring with unity $\exp(1)$ and componentwise product.
\end{theorem}
\begin{proof}
It is easy to see that $\exp(a)\ast\exp(b)=\exp(a+b)$. Then the operation $\ast$ is closed in
$\exp(R)$. It is also clear that $(\exp(R),\ast)$ is a commutative group with unit $\exp(0)$ and
inverse elements of the form $\exp(-a)$. On the other hand, $\exp(a)\cdot\exp(b)=\exp(ab)$ and the
operation $\cdot$ is closed in $\exp(R)$. Then $(\exp(R),\cdot)$ is a commutative monoid with unity.
Finally
\begin{eqnarray*}
\exp(a)\cdot(\exp(b)\ast\exp(c))&=&(\exp(a)\cdot\exp(b))\ast(\exp(a)\cdot\exp(c))
\end{eqnarray*}
and $\exp(R)$ is a commutative ring.
\end{proof}

As we will notice below, the ring $\exp(R)$ acts on the set $\Opa(\hurw_{R}[[x]])$

\begin{proposition}\label{prop_mult_esc}
Take $f\in\hurw_{R}[[x]]$ and $a \in R$. Then
\begin{equation}
\exp(a)\cdot\Opa(f)=\Opa(af).
\end{equation}
where the product $\exp(a)\cdot\Opa(f)$ is componentwise.
\end{proposition}
\begin{proof}
Let $a\neq0$. We will prove that $A_{n}(af)=a^{n}A_{n}(f)$. The proof will be by induction. For $n=0$, it is trivial. It is true for $n=1$, since $A_{1}(af)=af=aA_{1}(f)$. Now suppose that our statement holds for every number smaller than $n$ and we will consider
the case $n+1$. We have
\begin{eqnarray*}
A_{n+1}(af)&=&\sum_{k=1}^{n}\sum_{\vert p(n)\vert=k}\dfrac{n!}{j_{1}!\cdots j_{n}!}\left[\frac{aA_{1}}{1!}\right]^{j_{1}}\cdots\left[\frac{a^{n} A_{n}}{n!}\right]^{j_{n}}a\delta^{k}f\\
&=&\sum_{k=1}^{n}\sum_{\vert p(n)\vert=k}\frac{n!a^{j_{1}+2j_{2}+\cdots +nj_{n}+1}}{j_{1}!\cdots j_{n}!}\left[\frac{A_{1}}{1!}\right]^{j_{1}}\cdots \left[\frac{A_{n}}{n!}\right]^{j_{n}}\delta^{k}f\\
&=&a^{n+1}\sum_{k=1}^{n}\sum_{\vert p(n)\vert=k}\frac{n!}{j_{1}!\cdots j_{n}!}\left[\frac{A_{1}}{1!}\right]^{j_{1}}\cdots \left[\frac{A_{n}}{n!}\right]^{j_{n}}\delta^{k}f\\
&=&a^{n+1}A_{n+1}(f)
\end{eqnarray*}
Now set $a=0$. Then
$$\Opa(0)=\Opa(0f)=\exp(0)\Opa(f)$$
and our statement turns out to be true.
\end{proof}

With everything stated above, we are ready to show that $\Opa(\hurw_{R}[[x]])$ is a commutative 
group with sum $\boxplus$ and inverse elements $\exp(-1)\cdot\Opa(f)$

\begin{theorem}
$(\Opa(\hurw_{R}[[x]]),\boxplus)$ is a commutative group.
\end{theorem}
\begin{proof}
From Lemma \ref{lemma_ope_+}, $\Opa(f)\boxplus\Opa(g)\in\Opa(\hurw_{R}[[x]])$. Now
take $0\in\hurw_{R}[[x]]$. So it's easy to see that $\Opa(0)=(x,0,0,...)=\textbf{0}_{\Opa}$ and by the Lemma above
\begin{eqnarray*}
\Opa(f)\boxplus\textbf{0}_{\Opa}&=&\Opa(f)\boxplus\Opa(0)\\
&=&\Opa(f+0)\\
&=&\Opa(f)
\end{eqnarray*}
then $\textbf{0}_{\Opa}$ is the neutral element in $(\Opa(\hurw_{R}[[x]]),\boxplus)$. By the Proposition above it is easy to show that $\exp(-1)\Opa(f)$ is the additive inverse of $\Opa(f)$.
Finally the associativity and commutativity of $(\Opa(\hurw_{R}[[x]]),\boxplus)$ follow from those of $(\hurw_{R}[[x]],+)$.
\end{proof}

From the above it follows that the map $\lambda:(\Opa(\hurw_{R}[[x]]),\boxplus)\rightarrow(\hurw_{R}[[x]],+)$ is a group isomorphism. Next we will define a product in $\Opa(\hurw_{R}[[x]])$

\begin{definition}
We define the product $\odot$ in $\Opa(\hurw_{R}[[x]])$ as follows
\begin{equation}
\Opa(f)\odot\Opa(g)=\left(C_{n}(f,g)\right)_{n\in\N}
\end{equation}
with
\begin{eqnarray}
C_{0}(f,g)&=&x\\
C_{n}(f,g)&=&\sum_{l=1}^{n}\mathcal{A}_{l}(f)\mathcal{A}_{n-l+1}(g), \ \ n\geq1
\end{eqnarray}
where
\begin{equation}
\mathcal{A}_{l}(f)=\sum_{\vert p(n)\vert=l}\alpha_{\vert p(n)\vert} A_{j_{1}}(f)A_{j_{2}}(f)\cdots A_{j_{r}}(f)
\end{equation}
$p(n)=j_{1}+j_{2}+\cdots+j_{r}$ and the $\alpha_{\vert p(n)\vert}$ are suitable numbers.
\end{definition}

As noted, performing multiplications on $\Opa(\hurw_{R}[[x]])$ seems very hard with the definition above. However, the following lemma will allow us an easier way

\begin{lemma}\label{lemma_ope_prod}
$\Opa(f)\odot\Opa(g)=\Opa(fg)$ for all $f,g\in\hurw_{R}[[x]]$.
\end{lemma}
\begin{proof}
Set $A_{n}(f)\equiv A_{n}(f(x))$. Then
\begin{eqnarray*}
A_{1}(fg)&=&fg=A_{1}(f)A_{1}(g),\\
A_{2}(fg)&=&(ff^{\prime})(g^{2})+(f^{2})(gg^{\prime})=A_{2}(f)A_{1}^{2}(g)+A_{1}^{2}(f)A_{2}(g),\\
A_{3}(fg)&=&f^{3}(g(g^{\prime})^{2}+g^{2}g^{\prime\prime})+4f^{2}f^{\prime}g^{2}g^{\prime}+
(f(f^{\prime})^{2}+f^{2}f^{\prime\prime})g^{3}\\
&=&A_{1}^{3}(f)A_{3}(g)+4A_{1}(f)A_{2}(f)A_{1}(g)A_{2}(g)+A_{1}^{3}(g)A_{3}(f).
\end{eqnarray*}
Assume by induction that
\begin{equation}
A_{n}(fg)=\sum_{l=1}^{n}\mathcal{A}_{l}(f)\mathcal{A}_{n-l+1}(g).
\end{equation}
By the Lemma \ref{lemma_delta_opa}, $A_{n+1}(fg)=fg\delta A_{n}(fg)$. So using the induction hypothesis
\begin{eqnarray*}
A_{n+1}(fg)&=&fg\delta A_{n}(fg)\\
&=&fg\delta\left(\sum_{l=1}^{n}\mathcal{A}_{l}(f)\mathcal{A}_{n-l+1}(g)\right)\\
&=&fg\delta\left(\mathcal{A}_{1}(f)\mathcal{A}_{n}(g)+\sum_{l=2}^{n-1}\mathcal{A}_{l}(f)\mathcal{A}_{n-l+1}(g)+\mathcal{A}_{n}(f)\mathcal{A}_{n}(1)\right)\\
&=&f\delta\mathcal{A}_{1}(f)g\mathcal{A}_{n}(g)+f\mathcal{A}_{1}(f)g\delta\mathcal{A}_{n}(g)\\
&&+\sum_{l=2}^{n-1}[f\delta\mathcal{A}_{l}(f)g\mathcal{A}_{n-l+1}(g)+f\mathcal{A}_{l}(f)g\delta\mathcal{A}_{n-l+1}(g)]\\
&&+f\delta\mathcal{A}_{n}(f)g\mathcal{A}_{1}(g)+f\mathcal{A}_{n}(f)g\delta\mathcal{A}_{1}(g)
\end{eqnarray*}
Again we will use the Lemma \ref{lemma_delta_opa} and also that $A_{1}(f)=f$, $A_{1}(g)=g$. We have
\begin{eqnarray*}
f\delta\mathcal{A}_{1}(f)&=&f\delta A_{n}(f)=A_{n+1}(f)\\
f\delta\mathcal{A}_{n}(f)&=&f\delta A_{1}^{n}(f)=nA_{1}^{n-1}(f)f\delta A_{1}(f)=nA_{1}^{n-1}(f)A_{2}(f)\\
f\mathcal{A}_{l}(f)&=&A_{1}(f)\mathcal{A}_{l}(f),\ \ \ 1\leq l\leq n
\end{eqnarray*}
and
\begin{eqnarray*}
f\delta\mathcal{A}_{l}(f)&=&f\sum_{\vert p(n)\vert=l}\alpha_{\vert p(n)\vert}\sum_{k=1}^{l}A_{j_{1}}(f)\cdots \delta A_{j_{k}}(f)\cdots A_{j_{l}}(f)\\
&=&\sum_{\vert p(n)\vert=l}\alpha_{\vert p(n)\vert}\sum_{k=1}^{l}A_{j_{1}}(f)\cdots f\delta A_{j_{k}}(f)\cdots A_{j_{l}}(f)\\
&=&\sum_{\vert p(n)\vert=l}\alpha_{\vert p(n)\vert}\sum_{k=1}^{l}A_{j_{1}}(f)\cdots A_{j_{k}+1}(f)\cdots A_{j_{l}}(f).
\end{eqnarray*}
Likewise, $g\delta\mathcal{A}_{n-l+1}(g)$ is obtained simply by replacing $f$ by $g$ and by noting
that $n-l+1$ takes the same values as $l$ when it ranging between $1$ and $n$. Further we note that $A_{1}(f)\mathcal{A}_{l}(f)$ and $f\delta\mathcal{A}_{l}(f)$ are contained in $\mathcal{A}_{l+1}(f)$ and that $A_{1}(g)\mathcal{A}_{n-l+1}(g)$ and $g\delta\mathcal{A}_{n-l+1}(g)$ are contained in 
$\mathcal{A}_{n-l+2}(g)$ for $l$ ranging between $2$ and $n-1$. Finally we put all of the above together to obtain
\begin{eqnarray*}
A_{n+1}(fg)&=&A_{n+1}(f)A_{1}^{n+1}(g)+\sum_{l=2}^{n}\mathcal{A}_{l}(f)\mathcal{A}_{n-l+2}(g)+
A_{n+1}(g)A_{1}^{n+1}(f)\\
&=&\mathcal{A}_{1}(f)\mathcal{A}_{n+1}(g)+\sum_{l=2}^{n}\mathcal{A}_{l}(f)\mathcal{A}_{n-l+2}(g)+
\mathcal{A}_{1}(g)\mathcal{A}_{n+1}(l)\\
&=&\sum_{l=1}^{n+1}\mathcal{A}_{l}(f)\mathcal{A}_{n-l+2}(g)
\end{eqnarray*}
just as we wanted to show.
\end{proof}

Now we will show that $\Opa(\hurw_{R}[[x]])$ is a ring

\begin{theorem}
$(\Opa(\hurw_{R}[[x]]),\boxplus,\odot)$ is a commutative ring with neutral element 
$\textbf{0}_{\Opa}=(x,0,0,...)$ and unity $\textbf{1}_{\Opa}=(x,1,0,0,...)$. 
This ring will be called \textbf{autonomous ring} 
\end{theorem}
\begin{proof}
It was already shown that $(\Opa(\hurw_{R}[[x]]),\boxplus)$ is a commutative group. By the Lemma \ref{lemma_ope_prod}, $\Opa(f(x))\odot\Opa(g(x))$ is contained in $\Opa(\hurw_{R}[[x]])$. Also $\textbf{1}_{\Opa}$ is a unit with respect to $\odot$ in 
$\Opa(\hurw_{R}[[x]])$, since $\Opa(1)=\textbf{1}_{\Opa}$ and
\begin{eqnarray*}
\textbf{1}_{\Opa}\odot\Opa(f(x))&=&\Opa(1\cdot f(x))\\
&=&\Opa(f(x)).
\end{eqnarray*}
The associativity follows from the associativity of $\cdot$ in $\hurw_{R}[[x]]$. Finally
\begin{eqnarray*}
\Opa(f)\odot(\Opa(g)\boxplus\Opa(h))&=&\Opa(f)\odot\Opa(g+h)\\
&=&\Opa(f(g+h))\\
&=&\Opa(fg+fh))\\
&=&\Opa(fg)\boxplus\Opa(fh)\\
&=&[\Opa(f)\odot\Opa(g)]\boxplus[\Opa(f)\odot\Opa(h)]
\end{eqnarray*}
Then $(\Opa(\hurw_{R}[[x]]),\boxplus,\odot)$ is a ring with unity $\textbf{1}_{\Opa}$.
\end{proof}

From the above it follows that the map $\lambda$ is an isormorphism of the ring $(\hurw_{R}[[x]],+,\cdot)$ in the ring $(\Opa(\hurw_{R}[[x]]),\boxplus,\odot)$. With the following result we will show how to obtain the inverse elements in $\Opa(\hurw_{R}[[x]])$

\begin{theorem}
$(\Opa(\hurw_{R}[[x]]^{*}),\odot)$ is a commutative group.
\end{theorem}
\begin{proof}
We have that
\begin{eqnarray*}
\textbf{1}_{\Opa}&=&\Opa(1)\\
&=&\Opa(f(x)f^{-1}(x))\\
&=&\Opa(f(x))\odot\Opa(f^{-1}(x)).
\end{eqnarray*}
Then $\Opa(f(x))$ has an inverse with respect to $\odot$ and
$$\Opa(f(x))^{-1}=\Opa(f^{-1}(x)).$$ 
Thus $(\Opa(\hurw_{R}[[x]]*),\odot)$ is a group.
\end{proof}

Now it will be shown that $\Opa(\hurw_{R}[[x]])$ is an $\exp(R)$-algebra

\begin{theorem}\label{theo_exp}
$\Opa(\hurw_{R}[[x]])$ is an $\exp(R)$-algebra.
\end{theorem}
\begin{proof}
It is should show that
\begin{eqnarray}
\exp(a)[\Opa(f)\boxplus\Opa(g)]&=&\exp(a)\Opa(f)\boxplus\exp(a)\Opa(g)]\\
\exp(a)[\Opa(f)\odot\Opa(g)]&=&[\exp(a)\Opa(f)]\odot\Opa(g)=\Opa(f)\odot[\exp(a)\Opa(g)]
\end{eqnarray}
and this follows directly from $\Opa(a(f+g))$ and $\Opa(afg)$ along with Proposition \ref{prop_mult_esc}.
\end{proof}

In the following theorems we will discuss the structure of the autonomous ring.

\begin{theorem}\label{theo_delta_domi}
The ring $\Opa(\hurw_{R}[[x]])$ is an integral domain.
\end{theorem}
\begin{proof}
We know that $\hurw_{R}[[x]]$ is domain when $R$ is also domain. Now take $f$, $g$ from 
$\hurw_{R}[[x]]$. If $\Opa(f)\odot\Opa(g)=\Opa(fg)=\textbf{0}_{\Opa}$, then $fg=0$. Then either 
$f=0$ or $g=0$. Thus either $\Opa(f)=\textbf{0}_{\Opa}$ or $\Opa(g)=\textbf{0}_{\Opa}$ and 
$\Opa(\hurw_{R}[[x]])$ is a domain as well.
\end{proof}

Let $I$ be an ideal in $R$. From [2] we use the following notation for ideals in $\hurw_{R}[[x]]$:

\begin{eqnarray}
I+\left\langle x\right\rangle&=&\left\{f(x)=\sum_{n=0}^{\infty}a_{n}\frac{x^{n}}{n!}:a_{0}\in I\right\}\\
H_{I}[[x]]&=&\left\{f(x)=\sum_{n=0}^{\infty}a_{n}\frac{x^{n}}{n!}:a_{n}\in I\right\}
\end{eqnarray}

Let $\chi:\Opa(\hurw_{R}[[x]])\rightarrow\hurw_{R}[[x]]$ be a homomorphism defined by 
$\chi(\Opa(f(x)))=A_{1}[f(x)]=f(x)$. We have

\begin{proposition}\label{propo_iso}
Let $I$ be an ideal in $R$. Then
\begin{enumerate}
\item $\Opa(\hurw_{R}[[x]])/\chi^{-1}(I+\left\langle x\right\rangle)\simeq\hurw_{R}[[x]]/(I+\left\langle x\right\rangle)\simeq R/I$.
\item $\Opa(\hurw_{R}[[x]])/\Opa(\hurw_{I}[[x]])\simeq\Opa(\hurw_{R}[[x]]/\hurw_{I}[[x]])$.
\end{enumerate}
\end{proposition}
\begin{proof}
\begin{enumerate}
\item Let $J=I+\left\langle x\right\rangle$. Define the map $\psi:\Opa(\hurw_{R}[[x]])\rightarrow\hurw_{R}[[x]]/J$ by $\psi=\tau\circ\chi$, where $\tau:\hurw_{R}[[x]]\rightarrow\hurw_{R}[[x]]/J$ 
is the canonical map. Then $\psi$ is an overjective homomorphism with $Ker(\psi)=\phi^{-1}(J)$ 
and hence $\Opa(\hurw_{R}[[x]])/\chi^{-1}(J)\simeq\hurw_{R}[[x]]/J$. The remainder is a known result.
\item Since $\hurw_{R}[[x]]/\hurw_{I}[[x]]\simeq\hurw_{R/I}[[x]]$, then 
$\Opa(\hurw_{R}[[x]]/\hurw_{I}[[x]])\simeq\Opa(\hurw_{R/I}[[x]])$. Define the map
$\sigma:\Opa(\hurw_{R}[[x]])\rightarrow\Opa(\hurw_{R/I}[[x]])$ by 
$\sigma(\Opa(f))=\Opa(\overline{f})$, where
$\overline{f(x)}=\sum_{n=0}^{\infty}\overline{a}_{n}\frac{x^{n}}{n!}$. Then $\sigma$ is an
overjective homomorphism with $Ker(\sigma)=\Opa(\hurw_{I}[[x]])$. Then
$\Opa(\hurw_{R}[[x]])/\Opa(\hurw_{I}[[x]])\simeq\Opa(\hurw_{R/I}[[x]])$.
\end{enumerate}
\end{proof}

\begin{theorem}\label{theo_maxi}
If $M_{\alpha}$ is the set of maximal ideals of $\hurw_{R}[[x]]$, then 
$\Opa(M_{\alpha})$ is the set of maximal ideals of $\Opa(\hurw_{R}[[x]])$.
\end{theorem}
\begin{proof}
On the one hand, if $M$ is a maximal ideal of $\hurw_{R}[[x]]$, then $\Opa(M)=\chi^{-1}(M)$
is a maximal ideal of $\Opa(\hurw_{R}[[x]])$. On the other hand, let $\Opa(M)$ denote the maximal ideal of $\Opa(\hurw_{R}[[x]])$, then $\chi(\Opa(M))$ is an ideal of
$\hurw_{R}[[x]]$. Take $\Opa(h)\in\Opa(\hurw_{R}[[x]])\setminus\Opa(M)$. Then 
$(\Opa(M),\Opa(h))=\Opa(\hurw_{R}[[x]])$ and therefore there exists $\Opa(f)\in\Opa(M)$, 
$\Opa(g)\in\Opa(\hurw_{R}[[x]])$ such that $\textbf{1}_{\Opa}=\Opa(g)\odot\Opa(h)\boxplus\Opa(f)=\Opa(gh+f)=\Opa(1)$. Thus $1=gh+f\in(\chi(\Opa(M)),h)$ and it follows that $\chi(\Opa(M))$ is a maximal ideal of $\hurw_{R}[[x]]$.
\end{proof}

\begin{theorem}\label{the_prime}
If $P+\left\langle x\right\rangle$ and $\hurw_{P}[[x]]$ are prime ideals in $\hurw_{R}[[x]]$, then
$\Opa(P+\left\langle x\right\rangle)$ and $\Opa(\hurw_{P}[[x]])$ are prime ideals in 
$\Opa(\hurw_{R}[[x]])$.
\end{theorem}
\begin{proof}
If $\Opa(f)\odot\Opa(g)=\Opa(fg)\in\Opa(P+\left\langle x\right\rangle)$, then
$fg\in P+\left\langle x\right\rangle$ and therefore either $f\in P+\left\langle x\right\rangle$ or
$g\in P+\left\langle x\right\rangle$. Then either$\Opa(f)\in\Opa(P+\left\langle x\right\rangle)$
or $\Opa(g)\in\Opa(P+\left\langle x\right\rangle)$. Proof for the second part is the same
\end{proof}

The simplest autonomous ring arises when $R$ is a field. We denote this field by $\F$

\begin{theorem}\label{theo_ideal_princi}
Let $\F$ be a field. The ring $\Opa(\hurw_{\F}[[x]])$ is a domain of principal ideals being of the form
\begin{equation}
\left\langle\Opa(x^{n})\right\rangle=\Opa(x^{n})\odot\Opa(\hurw_{\F}[[x]])
\end{equation}
for all $n\geq1$ and with only ideal maximal
\begin{equation}
\left\langle\Opa(x)\right\rangle=\Opa(x)\odot\Opa(\hurw_{\F}[[x]])
\end{equation}
\end{theorem} 
\begin{proof}
It is a known fact that if $\F$ is a field, then $(\hurw_{\F}[[x]],+,\cdot)$ is a domain
of prime ideals with ideals $\left\langle x^{n}\right\rangle$, $n\geq1$. Since $\lambda$  is an 
is an isomorphism of $(\hurw_{\F}[[x]],+,\cdot)$ into $(\Opa(\hurw_{\F}[[x]]),\boxplus,\odot)$,
then it follows that $\left\langle\Opa(x^{n})\right\rangle$ is the set of ideals in
$\Opa(\hurw_{\F}[[x]])$ and likewise it is proved that $\left\langle\Opa(x)\right\rangle$ is its maximal ideal.
\end{proof}

\begin{theorem}\label{theo_campo}
Let $\F$ be a field. The set $\Opa[\hurw_{\F}((x))]=\Opa(\hurw_{\F}[[x]])/\left\langle\Opa(x)\right\rangle$ is a field.
\end{theorem}
\begin{proof}
Since $\left\langle\Opa(x)\right\rangle$ is maximal, it follows by the Proposition \ref{propo_iso}
that $\Opa[\hurw_{\F}((x))]$ is a field. 
\end{proof}

We end this section by defining polynomials in $\Opa(\hurw_{R}[[x]])$. We begin by defining
the monomials
\begin{equation}
\Opa(x)_{\odot}^{k}=\Opa(x)\odot\Opa(x)\odot\cdots\odot\Opa(x)
\end{equation}
and then we have the
\begin{definition}
A polynomial of degree $n$ in $\Opa(\hurw_{R}[[x]])$ is one of the form
\begin{equation}
\bbox_{k=0}^{n}\exp(a_{k})\Opa(x)_{\odot}^{k}
\end{equation}
\end{definition}
Since
\begin{equation}
\bbox_{k=0}^{n}\exp(a_{k})\Opa(x)_{\odot}^{k}=\bbox_{k=0}^{n}\Opa(a_{k}x^{k})=\Opa\left(\sum_{k=0}^{n}a_{k}x^{k}\right),
\end{equation}
then a polynomial in $\Opa(\hurw_{R}[[x]])$ is the image for $\lambda$ of the polynomial
$\sum_{k=0}^{n}a_{k}x^{k}$. Then the ring $R[x]$ is mapped to $\Opa(R[x])$. If the polynomial $f(x)$
has factorization $f(x)=\prod_{i=1}^{n}(x-a_{i})$, then
\begin{eqnarray*}
\Opa(f(x))&=&\Opa\left(\prod_{i=1}^{n}(x-a_{i})\right)\\
&=&\bcast_{i=1}^{n}\Opa(x-a_{i})\\
&=&\bcast_{i=1}^{n}(x,x-a_{i},x-a_{i},...)\\
&=&\bcast_{i=1}^{n}[(x-a_{i},x-a_{i},x-a_{i},...)+(a_{i},0,0,...)]\\
&=&\bcast_{i=1}^{n}[(x-a_{i})\textbf{1}+(a_{i},0,0,...)]
\end{eqnarray*}

where $\textbf{1}=(1,1,1,...)$. Then the polynomial $\Opa(f(x))$ factors into 
$\Opa(R[x])$.

Finally, let $\epsilon_{a}:R[x]\rightarrow R$ be the evaluation map given by 
$\epsilon_{a}(f(x))=f(a)$. Now extend $\epsilon_{a}$ to $\Opa(R[x])$ by defining the map
$\overline{\epsilon_{a}}:\Opa(R[x])\rightarrow\hurw_{R}$ by
$\overline{\epsilon_{a}}(\Opa(f(x))=\Opa(\epsilon_{a}f(x))=\Opa(f(a))$. Denote the ring
$(Im(\overline{\epsilon_{a}}),\boxplus,\odot)$ the image ring of $\Opa(R[x])$ by the map
$\overline{\epsilon_{a}}$. Then $\overline{\epsilon_{a}}$ is a homorphism of rings
with kernel $Ker(\overline{\epsilon_{a}})=\left\langle\Opa(x-a)\right\rangle$. Thus
$\Opa(R[x])/\left\langle\Opa(x-a)\right\rangle\simeq Im(\overline{\epsilon_{a}})$.

\section{One-dimensional flow over domains}

A one-dimensional continuous dynamical system is a tuple $(S,X,\phi)$ in which $S\subseteq\R$ is 
the set of times, $X\subseteq\R$ is the phase space and $\phi:S\times X\rightarrow X$ is 
the system flow satisfying

\begin{enumerate}
\item $\phi(0,x)=x$
\item $\phi(t,\phi(s,x))=\phi(t+s,x)$
\end{enumerate}

When $S=R$ and putting $\phi_{t}(x)=\phi(t,x)$ we notice that $\phi_{t}$ is a group acting on 
the phase space $X$. When we put $S=\Z$, the dynamical system is referred to as a
discrete dynamical system. In this section we define a dynamical system with $S=R$ a ring, 
$X=\hurw_{T}[[t]]$ where $T$ is some ring with coefficients in $\hurw_{R}[[x]]$. 

From [6] we know that the solution $\phi$ of (\ref{autonoma}) is given by
\begin{equation}
\phi(t,x)=x+\sum_{n=1}^{\infty}A_{n}(f(x))\dfrac{t^{n}}{n!}.
\end{equation}

This solution is a flow defined on $f(x)$. Then we can notice that the operator $\Opa$ maps $f(x)$ to the action group $\{\phi_{t}:t\in\R\}$.  With this in mind we have the following definition

\begin{definition}
We define a \textbf{dynamic system} over the ring $R$ as the tuple $(R,\hurw_{S}[[t]],\Phi)$ where 
$R$ is the \textbf{set of times}, $\hurw_{S}[[t]]$ the \textbf{phase space} and $\Phi$ is the \textbf{flow}. 
\begin{equation}\label{flujo_diferencial}
\Phi(t,x,f(x))=x+\sum_{n=1}^{\infty}A_{n}(f(x))\dfrac{t^{n}}{n!},
\end{equation}
that is, $\Phi$ is the map
$\Phi:R\times\hurw_{S}[[t]]\times\hurw_{R}[[x]]\rightarrow\hurw_{S}[[t]]$ where $S=\hurw_{R}[[x]]$.
\end{definition}

\begin{theorem}
$\Phi(t,x,f(x))$ is a flow.
\end{theorem}
\begin{proof}
When $t=0$ we have from (\ref{flujo_diferencial}) that $\Phi(0,x,f(x))=x$. We will now prove property 2) of a flow. On the one hand, the Taylor expansion of $\Phi(s+t,x,f(x))$ is $\sum_{n=0}^{\infty}\delta_{s}^{n}\Phi(s,x,f(x))\frac{t^{n}}{n!}$. On the other hand, making $\Phi_{s}=\Phi(s,x,f(x))$ we have
\begin{eqnarray*}
\Phi(t,\Phi(s,x,f(x)),f(x))&=&\Phi_{s}+\sum_{n=1}^{\infty}A_{n}(f(\Phi_{s}))\frac{t^{n}}{n!}\\
&=&\Phi_{s}+\sum_{n=1}^{\infty}\delta_{s}^{n}\Phi(s,x,f(x))\frac{t^{n}}{n!}\\
&=&\sum_{n=0}^{\infty}\delta_{s}^{n}\Phi(s,x,f(x))\frac{t^{n}}{n!}\\
&=&\Phi(s+t,x,f(x)).
\end{eqnarray*}
Then $\Phi(t,x,f(x))$ is a flow as stated.
\end{proof}

Clearly $\{\Phi_{t}:t\in R\}$, with $\Phi_{t}=\Phi(t,x,f(x))$, is a group acting on $\hurw_{S}[[t]]$. Let us define two derivatives on $\hurw_{S}[[t]]$, $\delta_{t}$ acting on the variable $t$ and $\delta_{x}=\delta$ acting on the variable $x$. We will use the Lemma \ref{lemma_delta_opa} to show a beautiful relation between $\delta_{t}Phi$ and $\delta_{x}\Phi$.

\begin{theorem}
Take $f(x)\in\hurw_{R}[[x]]$. Then
\begin{equation}
f(x)\delta_{x}\Phi(t,x,f(x))=\delta_{t}\Phi(t,x,f(x))=f(\Phi)
\end{equation}
\end{theorem}
\begin{proof}
From (\ref{flujo_diferencial}) and the Lemma \ref{lemma_delta_opa} we have
\begin{eqnarray*}
f(x)\delta_{x}\Phi(t,x,f(x))&=&f(x)\delta_{x}\left(x+\sum_{n=1}^{\infty}A_{n}(f(x))\dfrac{t^{n}}{n!}\right)\\
&=&f(x)+\sum_{n=1}^{\infty}f(x)\delta_{x}\left(A_{n}(f(x))\right)\dfrac{t^{n}}{n!}\\
&=&A_{1}([f(x)])+\sum_{n=1}^{\infty}A_{n+1}(f(x))\dfrac{t^{n}}{n!}\\
&=&\sum_{n=0}^{\infty}A_{n+1}(f(x))\dfrac{t^{n}}{n!}\\
&=&\delta_{t}\Phi(t,x,f(x)).
\end{eqnarray*}
\end{proof}

Next we will prove that the flows of $\delta_{t}\Phi=f(\Phi)$ and of $\delta_{t}Phi=af(\Phi)$,
with $a$in $R$, have matching trajectories. That is, 

\begin{theorem}
For all $a$in $R$ it is fulfilled that $\Phi(t,x,af(x))=\Phi(at,x,f(x))$
\end{theorem}
\begin{proof}
From the Proposition \ref{prop_mult_esc} we know that $\Opa(af(x))=\exp(a)\Opa(f(x))$ for all 
$a\in R$. For $a\neq0$ we have
\begin{eqnarray*}
\Phi(t,x,af(x))&=&x+\sum_{n=1}^{\infty}a^{n}A_{n}(f(x))\frac{t^{n}}{n!}\\
&=&x+\sum_{n=1}^{\infty}A_{n}(f(x))\frac{(at)^{n}}{n!}\\
&=&\Phi(at,x,f(x)).
\end{eqnarray*}
When $a=0$, $\Phi(t,x,0)=\Phi(0,x,f(x))=x$.
\end{proof}

The above theorem means that $\rho_{t}$ maps the class $\exp(R)\Opa(f)$ to the flow
$\Phi(t,x,f)$ for all $t\in R$. This is so because
$$\Phi(t,x,Rf(x))=\Phi(tR,x,f(x))=\Phi(R,x,f(x))$$
where $Rf(x)=\{af(x):a\in R\}$ is the set of scalar multiples of $f$.

When $I$ is a ideal in $R$, we have
$$\Phi(I,x,Rf(x))=\Phi(RI,x,f(x))=\Phi(I,x,f(x))$$

Now fix an $x$ in $R$ and $f\in\hurw_{R}[[x]]$. Define the homomorphism of groups 
$\sigma:R\rightarrow\Phi_{R}(x,f)$ by $\sigma(a)=\Phi_{a}(x,f)$, where we make $\Phi_{t}(x,f)\equiv\Phi(t,x,f)$. We want to extend the homomorphism $\sigma$ to a homomorphism of $R$-modules by showing that $\Phi_{R}(x,f)$ has precisely $R$-module structure.
By the group property of $\Phi_{R}(x,f)$, for all $n\in\N$  we have
$$\Phi_{nt}(x,f)=\Phi_{t}(x,f)\circ\Phi_{t}(x,f)\circ\cdots\circ\Phi_{t}(x,f)$$
Then we compose the flow $\Phi_{t}(x,f)$ with itself $n$ times and we can define 
$$n\star\Phi_{t}(x,f)=\Phi_{t}(x,f)\circ\Phi_{t}(x,f)\circ\cdots\circ\Phi_{t}(x,f).$$

In this way we can define the product $\star:R\times\Phi_{R}(x,f)\rightarrow\Phi_{R}(x,f)$ by
$$a\star\Phi_{t}(x,f)=\Phi_{at}(x,f)$$ 
for all $a\in R$ and it is very easy to observe that with the product $\star$ the subgroup $\Phi_{R}(x,f)$ adopts $R$-module structure and $\sigma$ becomes a homomorphism of $R$-module, since the ring $R$ is an $R$-module in itself. Now we want to understand under which conditions of the fixed element $x$ the $\sigma$-homomorphism becomes an isomorphism of $R$-module isomorphism.

Now fix a polynomial $f(x)\in R[x]$ and extend the map $\overline{\epsilon_{a}}$ to the flow $\Phi_{t}(x,f)$ by defining
\begin{equation}
\overline{\epsilon_{a}}\Phi_{t}(x,f)=\overline{\epsilon_{a}}\left(x+\sum_{n=1}^{\infty}A_{n}(f(x))\frac{t^{n}}{n!}\right)=a+\sum_{n=1}^{\infty}A_{n}(f(a))\frac{t^{n}}{n!}
\end{equation}

and let
\begin{equation}
\Gamma_{a}:=\{\overline{\epsilon_{a}}\Phi_{t}(x,f):t\in R\}
\end{equation}
be the orbit or trajectory of $a$. 

Then $\Gamma_{0}$ is the exponential generating function of the sequence 
$\overline{\epsilon_{0}}\Opa(f(x))$. If $\Gamma_{x_{0}}=\{x_{0}\}$, then $x_{0}$ is an equilibrium
point for $\Phi_{R}(x,f)$. Equilibrium points are obtained
when $A_{n}(f(x_{0}))=0$ for $n\geq1$, i.e., when $f(x_{0})=0$ for some $x_{0}\in R$. If $x_{0}$ 
is not an equilibrium point, then it will be called a regular point of $x_{0}$ regular point of 
$\Phi_{R}(x,f)$. 

Denote
\begin{eqnarray*}
\ker(\sigma)&=&\{t\in R:\sigma(t)=x\}\\
&=&\{t\in R:\Phi_{t}(x,f)=x\}
\end{eqnarray*} 
the kernel of $\sigma$. If $\overline{x}$ is an equilibrium point, then
$\sigma:R$ and $\ker(\sigma)=R$. If $x$ is a regular point, $\ker(\sigma)=R$ and $\sigma$ is injective. Since $\sigma$ is overjective, then $\sigma$ is an isomorphism of $R$-modules. We will say that the module $\Phi_{R}(x,f)$ is a trivial module if $x$ is an equilibrium point. Otherwise it would be called a nontrivial module.

Let $I$ be an ideal of $R$ and $x$ be a regular point. Then $\sigma(I)=\Phi_{I}(x,f)$ is a submodule in $\Phi_{R}(x,f)$. Then we can establish the following correspondence
\begin{eqnarray}
\{\text{Ideals $I$ in }R\}\Leftrightarrow\{\text{Submodules }\Phi_{I} \text{ of }\Phi_{R}(x,f)\}\Leftrightarrow\{y^{\prime}=If(y)\}
\end{eqnarray}
where $\{y^{\prime}=If(y)\}$ denotes the set of all autonomous differential equations $y^{\prime}=af(y)$, with $a\in I$.

We end this section with the following results
\begin{theorem}
A nontrivial $R$-module $\Phi_{R}(x,f)$ is a torsion-free cyclic module.
\end{theorem}
\begin{proof}
Suppose $x$ is a regular point. To show that $\Phi_{R}(x,f)$ is a cyclic module, it is sufficient 
to take a unit $u$ in $R$. Then $R\star\Phi_{u}(x,f)=\Phi_{uR}(x,f)=\Phi_{R}(x,f)$ and so 
$\Phi_{R}(x,f)$ is cyclic. Now we will show that $\Phi_{R}(x,f)$ is torsion free.
On the one hand, there exists an ideal $I\subset R$ such that $\Phi_{R}(x,f)$ is isomorphic to $R/I$.
Since it was already shown that $R$ and $\Phi_{R}(x,f)$ are isomorphic, then it follows that $I$ is
the zero ideal. On the other hand, $\Phi_{R}(x,f)$ is a cyclic $R$-module is equivalent to saying
that the multiplication homomorphism $\tau_{s}:R\rightarrow\Phi_{R}(x,f)$, 
$\tau_{s}(a)=a\star\Phi_{s}(x,f)$, is a homomorphism of overjective $R$-modules. Write 
$I=\ker(\tau_{s})$. By the first isomorphism theorem for module, $\overline{\tau_{s}}$ is 
an isomorphism from $R/I$ to $\Phi_{R}(x,f)$. Since $\overline{\tau_{s}}$ is the annihilator 
$\ann(\Phi_{s}(x,f))$ of $\Phi_{s}(x,f)$, then $\ann(\Phi_{s}(x, f))=I=0$ for any 
$\Phi_{s}(x,f)\in\Phi_{R}(x,f)$ and $0$ is the only torsion element in $\Phi_{R}(x,f)$.
\end{proof}

\section{Ring of one-dimensional flows}

Define the map $\rho_{t}:\Opa(\hurw_{R}[[x]])\rightarrow\hurw_{S}[[t]]$, where $S=\hurw_{R}[[x]]$, as $\rho_{t}\Opa(f(x))=x+\sum_{n=1}^{\infty}A_{n}(f(x))\frac{t^{n}}{n!}$. Then
\begin{eqnarray}
\Phi(t,x,f(x))&=&\rho_{t}\Opa(f(x)).
\end{eqnarray}

Then it is possible to extend the operations of the ring $\Opa(\hurw_{R}[[x]])$ to the set 
$\rho_{t}\Opa(\hurw_{R}[[x]])$. We have
\begin{definition}
Take $f(x)$ and $g(x)$ from $\hurw_{R}[[x]]$. We define the sum $\boxplus$ and the product
$\odot$ in $\rho_{t}\Opa(\hurw_{R}[[x]])$ as follows
\begin{eqnarray}
\rho_{t}\Opa(f)\boxplus\rho_{t}\Opa(g)&=&\rho_{t}[\Opa(f)\boxplus\Opa(g)]\\
\rho_{t}\Opa(f)\odot\rho_{t}\Opa(g)&=&\rho_{t}[\Opa(f)\odot\Opa(g)]
\end{eqnarray}
\end{definition}

Then we have the following result
\begin{theorem}
The set $\rho_{t}\Opa(\hurw_{R}[[x]])$ with sum $\boxplus$ and product $\odot$ is a commutative ring with units $\rho_{t}\Opa(0)=x$ and $\rho_{t}\Opa(1)=x+t$. The ring $\rho_{t}\Opa(\hurw_{R}[[x]])$ will be called \textbf{ring of one-dimensional flows}.
\end{theorem}
\begin{proof}
By the definition above the set $\rho_{t}\Opa(\hurw_{R}[[x]])$ inherits the ring properties of 
$\Opa(\hurw_{R}[[x]])$. On the other hand, $\rho_{t}\Opa(0)=\rho_{t}\textbf{0}_{\Opa}=x$ and $
\rho_{t}\Opa(1)=\rho_{t}\textbf{1}_{\Opa}=x+t$.
\end{proof}

Clearly $\rho_{t}$ is a ring isomorphism and therefore $\rho_{t}\Opa(\hurw_{R}[[x]])$ inherits all the algebraic properties of $\Opa(\hurw_{R}[[x]])$ from Section 3.
\begin{theorem}\label{theo_delta_domi_flujo}
$\rho_{t}\Opa(\hurw_{R}[[x]])$ is an integral domain.
\end{theorem}

Let $\chi:\rho_{t}\Opa(\hurw_{R}[[x]])\rightarrow\hurw_{R}[[x]]$ be an homomorphism defined by
$\chi(\rho_{t}\Opa(f(x)))=f(x)$. We have the proposition

\begin{proposition}\label{propo_iso_flujo}
Let $I$ be an ideal in $R$. Then
\begin{enumerate}
\item $\rho_{t}\Opa(\hurw_{R}[[x]])/\chi^{-1}(I+\left\langle x\right\rangle)\simeq\hurw_{R}[[x]]/(I+\left\langle x\right\rangle)\simeq R/I$.
\item $\rho_{t}\Opa(\hurw_{R}[[x]])/\rho_{t}\Opa(\hurw_{I}[[x]])\simeq\rho_{t}\Opa(\hurw_{R}[[x]]/\hurw_{I}[[x]])$.
\end{enumerate}
\end{proposition}

\begin{theorem}\label{theo_maxi_flujo}
If $M_{\alpha}$ is the set of maximal ideals of $\hurw_{R}[[x]]$, then $\rho_{t}\Opa(M_{\alpha})$ 
is the set of maximal ideals of $\rho_{t}\Opa(\hurw_{R}[[x]])$.
\end{theorem}

\begin{theorem}\label{the_prime_flujo}
If $P+\left\langle x\right\rangle$ and $\hurw_{P}[[x]]$ are prime ideals in
$\hurw_{R}[[x]]$, then $\rho_{t}\Opa(P+\left\langle x\right\rangle)$ and
$\rho_{t}\Opa(\hurw_{P}[[x]])$ are prime ideals in $\rho_{t}\Opa(\hurw_{R}[[x]])$.
\end{theorem}

\begin{theorem}\label{theo_campo_flujo}
Let $\F$ be a field. Then the set
$$\rho_{t}\Opa[\hurw_{F}((x))]=\rho_{t}\Opa(\hurw_{F}[[x]])/\left\langle\rho_{t}\Opa(x)\right\rangle$$ is a field.
\end{theorem}

The reason for constructing the ring $\rho_{t}\Opa(\hurw_{R}[[x]])$ is because it contains all the solutions to the autonomous differential equations $\delta_{t}\Phi=f(\Phi)$ for each function $f(x)$ in $\hurw_{R}[[x]]$. In this ring it will be possible to decompose the solutions 
of an autonomous differential equation of order one into simpler solutions.\\ 
First suppose that $f=f_{1}+f_{2}+\cdots+f_{k}$. We wish to solve the differential equation $\delta_{t}\Phi=f(\Phi)$.  The flow of this equation becomes
\begin{eqnarray*}
\Phi(t,x,f(x))&=&\rho_{t}\Opa(f(x))\\
&=&\rho_{t}\Opa\left(\sum_{i=1}^{k}f_{i}(x)\right)\\
&=&\rho_{t}\left(\bbox_{i=1}^{k}\Opa(f_{i}(x))\right)\\
&=&\bbox_{i=1}^{k}\rho_{t}\Opa(f_{i}(x))\\
&=&\bbox_{i=1}^{k}\Phi(t,x,f_{i}(x)).
\end{eqnarray*} 
Thus the flow of $\delta_{t}\Phi=f(\Phi)$ decomposes into summands where each summand is the flow of the equation $\delta_{t}\Phi=f_{i}(\Phi)$.

\begin{example}
Find the solution to the differential equation $y^{\prime}=e^{y}+sin(y)$, where $e^{ax}$ and $\sin(x)$ are in $\hurw_{\R}[[x]]$. We first calculate the flow of $y^{\prime}=e^{ay}$. We have
\begin{eqnarray*}
\rho_{t}\Opa(e^{ax})&=&\rho_{t}(x,e^{ax},ae^{2ax},2!a^{2}e^{3ax},...)\\
&=&x+\sum_{n=1}^{\infty}(n-1)!a^{n-1}e^{anx}\frac{t^{n}}{n!}\\
&=&x+\frac{1}{a}ln\left(\frac{1}{1-ate^{ax}}\right)
\end{eqnarray*}
In this way 
\begin{equation}
\Phi(t,x,e^{ax})=x+\frac{1}{a}ln\left(\frac{1}{1-ate^{ax}}\right)
\end{equation} 
is the flow we are looking for. Now we only have to use the previous result with $a=1,i,-i$. In this
way
\begin{small}
\begin{eqnarray*}
\Phi(t,x,e^{x}+\sin(x))&=&\Phi\left(t,x,e^{x}+\frac{1}{2i}e^{ix}-\frac{1}{2i}e^{-ix}\right)\\
&=&\Phi(t,x,e^{x})\boxplus\Phi(t,x,\frac{1}{2i}e^{ix})\boxplus\Phi(t,x,-\frac{1}{2i}e^{-ix})\\
&=&\left[x+ln\left(\frac{1}{1-te^{x}}\right)\right]\boxplus\left[x-(i)ln\left(\frac{2}{2-te^{ix}}\right)\right]\\
&&\boxplus\left[x+(i)ln\left(\frac{2}{2+te^{-ix}}\right)\right]
\end{eqnarray*}
\end{small}
is the flow of the equation $y^{\prime}=e^{y}+sin(y)$.
\end{example}

Now suppose that $f$ factorizes as $f=f_{1}f_{2}\cdots f_{k}$ in $\hurw_{R}[[x]]$ and let's find the solution to the equation $y^{\prime}=f(y)$. Then its flow will be 
\begin{eqnarray*}
\Phi(t,x,f(x))&=&\bcast_{i=1}^{k}\rho_{t}\Opa(f_{i})\\
&=&\bcast_{i=1}^{k}\Phi(t,x,f_{i}(x)).
\end{eqnarray*}
where each $\Phi(t,x,f_{i}(x))$ is the flow of the equation 
$y^{\prime}=f_{i}(y)$ associated with the dynamical system $(R,\hurw_{S}[[x]],\rho_{t}\Opa(f_{i}))$.

Now suppose $f(x)=\sum_{k=0}^{n}a_{k}x^{k}$ is a polynomial of degree $n$ in $R[x]$. We want to define a polynomial in the ring $\rho_{t}\Opa(R[x])$. Let us first define the monomials in $\rho_{t}\Opa(R[x])$ as
\begin{eqnarray}\label{eqn_pot_rho}
[\rho_{t}\Opa(x)]_{\odot}^{k}&=&\overbrace{\rho_{t}\Opa(x)\odot\cdots\odot\rho_{t}\Opa(x)}^{k}
\end{eqnarray}

and let's calculate its value
\begin{lemma}\label{lemma_potencia}
When $k=0$ we have that $[xe^{at}]_{\odot}^{0}=x+at$ and when $k=1$, clearly $[xe^{at}]_{\odot}^{1}=xe^{at}$. For $k\geq2$ we have
\begin{equation}
[xe^{at}]_{\odot}^{k}=\frac{x}{\sqrt[k-1]{1-a(k-1)x^{k-1}t}}
\end{equation}
\end{lemma}
\begin{proof}
Por (\ref{eqn_pot_rho}), $[xe^{at}]_{\odot}^{k}=\rho_{t}\Opa(ax^{k})$. Then
we need to find the flow of the equation $y^{\prime}=ay^{k}$, for $k\geq0$. When $k=0$,
the equation $y^{\prime}=a$ has solution $\Phi(t,x,a)=x+at$. When $k=1$, the equation
$y^{prime}=ay$ has solution $\Phi(t,x,ax)=xe^{at}$. For $k\geq2$ we will use the method of 
separation of variables. Then the flow is
\begin{equation*}
\Phi(t,x,ax^{k})=\frac{x}{\sqrt[k-1]{1-a(k-1)x^{k-1}t}}.
\end{equation*}
\end{proof}

Then, by direct application of this Lemma, we have the following results
\begin{definition}
A polynomial of degree $n$ in $\rho_{t}\Opa(R[x])$ is of the form
\begin{equation}
\bbox_{k=0}^{n}[\rho_{t}\Opa(a_{k}x)]_{\odot}^{k}=(x+a_{0}t)\boxplus xe^{a_{1}t}\boxplus\bbox_{k=2}^{n}\frac{x}{\sqrt[k-1]{1-a_{k}(k-1)x^{k-1}t}}
\end{equation}
\end{definition}

Now let $\F$ be a field and let $f(x)$ be a polynomial of degree $n$ in $\F[x]\subset\hurw_{\F}[[x]]$. Further suppose $f(x)$ factorizes as $\prod_{i=1}^{n}(x-a_{i})$ in $\F[x]$. Then
\begin{eqnarray*}
\bbox_{k=0}^{n}[\rho_{t}\Opa(b_{k}x)]_{\odot}^{k}
&=&\rho_{t}\Opa\left(\prod_{i=1}^{n}(x-a_{i})\right)\\
&=&\rho_{t}\left(\bcast_{i=1}^{n}\Opa(x-a_{i})\right)\\
&=&\bcast_{i=1}^{n}\rho_{t}\Opa(x-a_{i})\\
&=&\bcast_{i=1}^{n}[(x-a_{i})e^{t}+a_{i}]
\end{eqnarray*}

\begin{theorem}
The irreducible polynomials in $\rho_{t}\Opa(\R[x])$ are of the form
\begin{equation}
(x-a)e^{t}+a
\end{equation}
\begin{equation}
\sqrt{d}\left(\frac{x-b/2+\sqrt{d}\tan(\sqrt{d}t)}{\sqrt{d}-(x-b/2)\tan(\sqrt{d}t)}\right)+\frac{b}{2}
\end{equation}
where $d=\frac{4c-b^{2}}{4}$.
\end{theorem}
\begin{proof}
We know that every irreducible polynomial in $\R[x]$ is of the form $x-a$ and $x^{2}-bx+c$. Then every polynomial in $\R[x]$ contains these factors. By solving the equations $y^{\prime}=y-a$ and $y^{\prime}=y^{2}-by+c$ we obtain the functions above.
\end{proof}

\begin{example}
Solve the equation $y^{\prime}=1-y+y^{2}-y^{3}$. Directly by the Lemma
\ref{lemma_potencia} la solución es
\begin{eqnarray*}
\Phi(t,x,f(x))&=&(x+t)\boxplus xe^{-t}\boxplus\left(\frac{x}{1-xt}\right)\boxplus\left(\frac{x}{\sqrt{1+2x^{2}t}}\right)
\end{eqnarray*}
By factoring $f(x)$ we obtain $f(x)=1-x+x^{2}-x^{3}=(1-x)(x^{2}+1)$. Then
\begin{eqnarray*}
\Phi(t,x,f(x))&=&\rho_{t}\Opa(1-x)\odot\rho_{t}\Opa(x^{2}+1)\\
&=&\rho_{t}\Opa(-(x-1))\odot\rho_{t}\Opa(x^{2}+1)\\
&=&[(x-1)e^{-t}-1]\odot\left(\frac{x+\tan(t)}{1-x\tan(t)}\right)
\end{eqnarray*}
would be the factored solution. A solution with the variable separation method leads to a function
that cannot be written explicitly. Then solve on the ring $\rho_{t}\Opa(\hurw_{\R}[[x]])$ is 
the best way to go.
\end{example}

As something very easy to prove we have the following
\begin{theorem}
The irreducible polynomials in $\rho_{t}\Opa(\C[x])$ are of the form
\begin{equation}
\Phi(t,x,x-a)=(x-a)e^{t}+a
\end{equation}
\end{theorem}

Other important results that can be deduced from the above Lemma are as follows
\begin{theorem}
Let $\F$ be a field. Then the ideals in $\rho_{t}\Opa(\hurw_{\F}[[x]])$ are of the form
\begin{equation}
\left\langle\rho_{t}(\Opa(x))\right\rangle=xe^{t}\odot\rho_{t}\Opa(\hurw_{\F}[[x]])
\end{equation}
and
\begin{equation}
\left\langle\rho_{t}(\Opa(x))_{\odot}^{m}\right\rangle=\frac{x}{\sqrt[m-1]{1-(m-1)x^{m-1}t}}\odot\rho_{t}\Opa(\hurw_{\F}[[x]])
\end{equation}
for $m\geq2$, where $\left\langle xe^{t}\right\rangle$ is a maximal ideal.
\end{theorem}
\begin{proof}
It follows by taking into account that $\left\langle\Opa(x)_{\odot}^{m}\right\rangle$ are ideals in $\Opa(\hurw_{\F}[[x]])$.
\end{proof}

\begin{theorem}
Let $\F$ be a field. In the field $\rho_{t}\Opa[\hurw_{\F}((x))]$ we have the following identities
\begin{eqnarray}
(x+at)&\odot&\left(x+\frac{t}{a}\right)=x+t,\ \ \ a\neq0\\
xe^{t}&\odot&\sqrt{2t+x^{2}}=x+t\\
\frac{x}{\sqrt[m-1]{1-(m-1)x^{m-1}t}}&\odot&\sqrt[m+1]{(m+1)t+x^{m+1}}=x+t
\end{eqnarray}
for $m\geq2$.
\end{theorem}
\begin{proof}
We have to keep in mind that the flows $\{rho_{t}\Opa(f)$ and $\rho_{t}\Opa(g)$ are inverses in 
the ring $\rho_{t}\Opa[\hurw_{\F}((x))]$ only if $fg=1$. Then the identities above arise from 
solving the differential equations $y^{\prime}=a$, $y^{\prime}=a^{-1}$, $y^{\prime}=y^{m}$ and
$y^{\prime}=y^{-m}$ for $m\geq1$, respectively.
\end{proof}

We end this section with the following result that relates the operations $\circ$, $\boxplus$ and 
$\odot$
\begin{theorem}
In the ring $\rho_{t}\Opa(\hurw_{R}[[x]])$ we have the following identities relating the operations $\circ$, $\boxplus$ and $\odot$.
\begin{enumerate}
\item $\Phi_{t}(x,f+g)\circ\Phi_{s}(x,f+g)=\Phi_{t+s}(x,f)\boxplus\Phi_{t+s}(x,g)$.
\item $\Phi_{t}(x,fg)\circ\Phi_{s}(x,fg)=\Phi_{t+s}(x,f)\odot\Phi_{t+s}(x,g)$
\end{enumerate} 
\end{theorem}
\begin{proof}
The proof is straightforward.
\end{proof}

\Addresses
\end{document}